\documentclass{amsart}
\usepackage{amsmath}
  \usepackage{paralist}
  \usepackage{amssymb}
 \usepackage{amsthm}
\usepackage[colorlinks=true]{hyperref}
\hypersetup{urlcolor=blue, citecolor=red}

\newtheorem{theorem}{Theorem}[section]

\newtheorem{lemma}[theorem]{Lemma}

\theoremstyle{definition}

\newtheorem{remark}[theorem]{Remark}

\numberwithin{equation}{section}

\title[Existence of normalized solutions]
      {Existence of normalized solutions to Choquard equation with general mixed nonlinearities}

\author[Meiling Zhu, Xinfu Li]{}


\thanks{*Corresponding author. Email Addresses: zhumeiling12365@163.com;  lxylxf@tjcu.edu.cn.}

\begin{document}
\maketitle

\centerline{\scshape Meiling Zhu$^a$,\  Xinfu Li$^{b*}$}

\medskip
{\footnotesize
\centerline{$^\mathrm{a}$College of Computer Science and
Engineering, Cangzhou Normal University,} \centerline{Cangzhou,
Hebei, 061001, P. R. China} \centerline{$^b$School of Science, Tianjin University of Commerce, Tianjin, 300134,
P. R. China}}

\bigskip

\begin{abstract}
We study the existence of normalized solutions to the following Choquard equation with $F$  being a Berestycki-Lions type function
\begin{equation*}
\begin{cases}
-\Delta u+\lambda u=
(I_{\alpha}\ast F(u))f(u),\quad \text{in}\  \mathbb{R}^N, \\
\int_{\mathbb{R}^N}|u|^2dx=\rho^2,
\end{cases}
\end{equation*}
where $N\geq 3$, $\rho>0$ is assigned,  $\alpha\in (0,N)$, $I_{\alpha}$ is the
Riesz potential, and $\lambda\in \mathbb{R}$ is an unknown parameter that appears as a
Lagrange multiplier. Here, the general nonlinearity $F$ contains the $L^2$-subcritical and  $L^2$-supercritical mixed case,  the Hardy-Littlewood-Sobolev lower critical and upper critical cases.
\end{abstract}

\section{Introduction and main results}

\setcounter{section}{1}
\setcounter{equation}{0}

We consider the following Choquard equation with an $L^2$-constraint
\begin{equation}\label{e1.3}
\begin{cases}
-\Delta u+\lambda u=
(I_{\alpha}\ast F(u))f(u),\quad \text{in}\  \mathbb{R}^N, \\
\int_{\mathbb{R}^N}|u|^2dx=\rho^2,
\end{cases}
\end{equation}
where $N\geq 3$, $\rho>0$ is assigned, $\lambda\in \mathbb{R}$ is an unknown parameter that appears as a
Lagrange multiplier,  and $I_{\alpha}$ is the
Riesz potential of order $\alpha\in (0,N)$ defined for every $x\in \mathbb{R}^N \setminus
\{0\}$ by
\begin{equation*}
I_{\alpha}(x):=\frac{A_\alpha(N)}{|x|^{N-\alpha}},\
A_\alpha(N):=\frac{\Gamma(\frac{N-\alpha}{2})}{\Gamma(\frac{\alpha}{2})\pi^{N/2}2^\alpha}
\end{equation*}
with $\Gamma$ denoting the Gamma function (see \cite{Riesz1949AM},
P.19). Here the nonlinearity $F(t)=b|t|^{\frac{N+\alpha}{N}}+G(t)$ with $b=0$ or $1$, $g(t)=G'(t)$, $f(t)=F'(t)$ and $G\in C^1(\mathbb{R},\mathbb{R})$ satisfies the following conditions $(G_1)$-$(G_3)$:\\
($G_1$)  $|g(t)|\lesssim |t|^{\frac{\alpha}{N}}+|t|^{\frac{2+\alpha}{N-2}},\ \forall\ t\in\mathbb{R}.$\\
($G_2$) \begin{equation*}\left\{\begin{array}{ll}
\lim_{t\to 0}\frac{G(t)}{|t|^{1+\frac{2+\alpha}{N}}}=+\infty,& \text{if}\ b=0,\\
\lim_{t\to 0}\frac{G(t)}{|t|^{1+\frac{4+\alpha}{N}}}=+\infty,& \text{if}\ b=1.
\end{array}
\right.
\end{equation*}
($G_3$) $\lim_{t\to 0}\frac{G(t)}{|t|^{\frac{N+\alpha}{N}}}=0$.

The equation (\ref{e1.3}) has several physical origins.
When $N = 3$, $\alpha = 2$ and $F(u)=|u|^2$,
it was investigated by Pekar in \cite{Pekar 1954} to study
the quantum theory of a polaron at rest. In \cite{Lieb 1977},
Choquard applied it as an approximation to Hartree-Fock theory of
one component plasma. It also arises in multiple particles systems
\cite{Gross 1996} and quantum mechanics \cite{Penrose 1996}. It is also known as the Schr\"{o}dinger-Newton equation or the Hartree equation. For fixed $\lambda>0$, the Choquard equation (\ref{e1.3}) has received many attentions in recent years, see \cite{Moroz-Van Schaftingen 2015} and the references therein.
However, physicists are more preferable to find normalized solutions of problem (\ref{e1.3}), i.e., the $L^2$-norm of solutions is given in advance. It then turns out to be a constrained problem restricted on a prescribed sphere in $L^2(\mathbb{R}^N)$. Consequently, a normalized solution of the Choquard equation always means a couple $(u, \lambda)$ which satisfies equation (\ref{e1.3}) with $\lambda$ appearing as a Lagrange multiplier.

When studying Choquard equation with power type nonlinearity in $H^1(\mathbb{R}^N)$
\begin{equation}\label{e1.12}
-\Delta u+\lambda u=(I_\alpha\ast|u|^p)|u|^{p-2}u,\ x\in
\mathbb{R}^{N},
\end{equation}
it usually requires $p\in [\underline{p},\bar{p}]$ by the Hardy-Littlewood-Sobolev inequality, where
$\bar{p}:=\frac{N+\alpha}{N-2}$ and $\underline{p}:=\frac{N+\alpha}{N}$ are the Hardy-Littlewood-Sobolev upper critical exponent and lower critical exponent, respectively. To consider normalized solutions of (\ref{e1.12}), the $L^2$-critical exponent $p^*:=1+\frac{2+\alpha}{N}$ appears as a special role. Since for $ p<p^*$, the functional is bounded below on the prescribed sphere in $L^2(\mathbb{R}^N)$, while for $p>p^*$, the functional is no longer bounded on the sphere, see \cite{{Cazenave-Lions
1982},{Luo 2019},{Ye 2016}} for more details.
Recently, the authors in \cite{{Cao-Jia-Luo-JDE},{Ding-Wang-2210},{Jia-Luo-CV},{Long-Li-Zhu},{Wang-Ma-Liu}} studied the normalized solutions to the Choquard equation
\begin{equation*}
-\Delta u+\lambda u=(I_\alpha\ast
|u|^{p})|u|^{p-2}u+ \mu(I_\beta\ast
|u|^{q})|u|^{q-2}u,\
x\in \mathbb{R}^N
\end{equation*}
with $(p,q)$ being at different kinds of intervals, and \cite{{Li2022},{Li21},{Li-Bao-Tang-23},{Yao-Chen22}}  considered
\begin{equation}\label{e11.3}
-\Delta u+\lambda u=
(I_{\alpha}\ast|u|^{p})|u|^{p-2}u+\mu
|u|^{q-2}u,\quad x\in\mathbb{R}^N
\end{equation}
with mixed nonlinearities.

As to the Choquard equation with general nonlinearity, for $\lambda>0$ being fixed, Moroz and Van Schaftingen \cite{Moroz-Van Schaftingen 2015} considered
 \begin{equation}\label{e2.1}
-\Delta u+\lambda u=
(I_{\alpha}\ast F(u))f(u),\quad x\in\mathbb{R}^N
\end{equation}
when $F$ is a Berestycki-Lions type function under the following general assumptions:\\
(F1) $F\in C^1(\mathbb{R}, \mathbb{R})$;\\
(F2) there exists $C>0$ such that for every $t\in \mathbb{R}$,
\begin{equation*}
|tf(t)|\leq C(|t|^{\frac{N+\alpha}{N}}+|t|^{\frac{N+\alpha}{N-2}});
\end{equation*}
(F3)\ $\lim_{t\to0}\frac{F(t)}{|t|^{\frac{N+\alpha}{N}}}=0,\ \lim_{|t|\to+\infty}\frac{F(t)}{|t|^{\frac{N+\alpha}{N-2}}}=0;$\\
(F4) there exists $t_0\in \mathbb{R}\backslash\{0\}$ such that $F(t_0)\neq0$.

Concerning the constrained problem (\ref{e1.3}) with general nonlinearity,
\cite{{Cingolani-Gallo-Tanaka-CV-22},{Cingolani-Tanaka21}} considered the $L^2$-subcritical case, and \cite{{Bartsch-Liu-Liu_2020},{Li-Ye_JMP_2014},{Xia-Zhang-JDE-2023},{Xu-Ma-2022},{Yuan-Chen-Tang_2020}} considered the $L^2$-supercritical case.
Precisely, Cingolani and Tanaka \cite{Cingolani-Tanaka21} obtained a normalized solution to (\ref{e1.3}) assuming that $F$ satisfies (F1), (F4) and it is $L^2$-subcritical, namely\\
(F5) there exists $C>0$ such that for every $t\in \mathbb{R}$,
\begin{equation*}
|tf(t)|\leq C(|t|^{\frac{N+\alpha}{N}}+|t|^{\frac{N+\alpha+2}{N}});
\end{equation*}
(F6)\ $\lim_{t\to0}\frac{F(t)}{|t|^{\frac{N+\alpha}{N}}}=0,\ \lim_{|t|\to+\infty}\frac{F(t)}{|t|^{\frac{N+\alpha+2}{N}}}=0.$\\
Xu and Ma \cite{Xu-Ma-2022} obtained a normalized solution to (\ref{e1.3}) assuming that $F$ satisfies (F1) and it is $L^2$-supercritical, namely\\
(F7)\ $\lim_{t\to0}\frac{f(t)}{|t|^{\frac{\alpha+2}{N}}}=0,\ \lim_{|t|\to+\infty}\frac{F(t)}{|t|^{\frac{N+\alpha+2}{N}}}=+\infty;$\\
(F8)\ $\lim_{|t|\to+\infty}\frac{F(t)}{|t|^{\frac{N+\alpha}{N-2}}}=0;$\\
(F9)\ $f(t)t<\frac{N+\alpha}{N-2}F(t)$ for all $t\in\mathbb{R}\backslash\{0\}$ and the map $t\mapsto \tilde{F}(t)/|t|^{1+\frac{\alpha+2}{N}}$ is
strictly decreasing on $(-\infty, 0)$ and strictly increasing on $(0,+\infty)$, where $\tilde{F}(t):=f(t)t-\frac{N+\alpha}{N}F(t)$. \\
However, there is no result considering the $L^2$-subcritical and  $L^2$-supercritical mixed case. In this paper, we will settle this open problem. Meanwhile, we handle the Hardy-Littlewood-Sobolev upper critical exponent
and lower critical exponent simultaneously.

Our first result is the following.

\begin{theorem}\label{thm1.1}
Let $N\geq 3$, $\alpha\in (0,N)$, $\rho>0$, and  $G$ satisfy ($G_1$)-($G_3$). Then there exits $\rho_0>0$ such that if $\rho<\rho_0$, there exist $u\in H^1(\mathbb{R}^N)$ and $\lambda_u>0$ such that $(u,\lambda_u)$ is a solution to  (\ref{e1.3}).
\end{theorem}

\begin{remark}\label{rmk1.1}
(1) The expression of $\rho_0$ in Theorem \ref{thm1.1}  can be written explicitly, see (\ref{e3.1}).\\
(2) The solution $u$ obtained in Theorem \ref{thm1.1} is a local minimizer, see Theorem \ref{thm3.1}.
\end{remark}

\begin{remark}\label{rmk1.3}
(1) Theorem \ref{thm1.1} is also right, by substituting $F$ with $-F$.\\
(2) In  view of  the conditions in \cite{{Cingolani-Tanaka21},{Moroz-Van Schaftingen 2015},{Xu-Ma-2022}}, we see that the conditions ($G_1$) and ($G_3$)  are very weak. The condition ($G_2$) is used in Lemma \ref{lem3.3} to obtain the upper bound of the functional.  ($G_2$) is also used in Cingolani, Gallo and Tanaka \cite{Cingolani-Gallo-Tanaka-CV-22} to obtain the multiplicity of normalized  solutions to (\ref{e1.3}) in the $L^2$-subcritical case, see (\cite{Cingolani-Gallo-Tanaka-CV-22}, Remark 4).\\
(3) It is obvious that the nonlinearities in  this paper contain the $L^2$-subcritical and  $L^2$-supercritical mixed case,  the Hardy-Littlewood-Sobolev lower critical
and upper  critical cases.
\end{remark}

It makes sense to ask whether the local minimizer found in Theorem \ref{thm1.1} is a ground state solution.
This is indeed the case under the following additional assumption $(G_4)$:\\
$(G_4)$ for every $u\in \mathcal{S}_\rho$, the function $(0,+\infty)\ni\tau\mapsto \varphi(\tau):=E(u_\tau)\in\mathbb{R}$ has at most one local maximum point $\tau_{u}$ and $\varphi'(\tau)<0$ for $\tau\in (\tau_{u},+\infty)$, where $$u_{\tau}(x):=\tau^{\frac{N}{2}}u(\tau x)\  \text{for}\  \tau>0,$$ $$\mathcal{S}_{\rho}:=\left\{u\in
H^1(\mathbb{R}^N):\|u\|_2=\rho\right\},$$
and  $E:H^1(\mathbb{R}^N)\to \mathbb{R}$ is the energy functional associated with (\ref{e1.3}) given by
\begin{equation}\label{e1.10}
\begin{split}
E(u):&=\frac{1}{2}\int_{\mathbb{R}^N}|\nabla u|^2dx-\frac{1}{2}\int_{\mathbb{R}^N}(I_\alpha\ast(F(u)))F(u)dx\\
&=\frac{1}{2}\int_{\mathbb{R}^N}|\nabla u|^2dx-\frac{1}{2}\int_{\mathbb{R}^N}[I_\alpha\ast(b|u|^{\frac{N+\alpha}{N}}+G(u))][b|u|^{\frac{N+\alpha}{N}}+G(u)]dx.
\end{split}
\end{equation}

\begin{theorem}\label{thm1.2}
Let the assumptions in Theorem  \ref{thm1.1}  hold and further assume $(G_4)$ holds. Then the  solution obtained  in  Theorem  \ref{thm1.1}  is a ground state solution to (\ref{e1.3}).  A ground state $v$ of (\ref{e1.3}) is defined as
\begin{equation*}
E|_{\mathcal{S}_{\rho}}'(v)=0\ \text{and}\ E(v)=\inf\{E(w):\ w\in \mathcal{S}_{\rho},\
E|_{\mathcal{S}_{\rho}}'(w)=0\}.
\end{equation*}
\end{theorem}

\begin{remark}\label{rmk1.4}
An example that satisfies $(G_4)$ is $F(t)=\nu |t|^p+\mu |t|^q$ with $\nu,\mu\in \mathbb{R}$, $\frac{N+\alpha}{N}<p<1+\frac{2+\alpha}{N}<q\leq \frac{N+\alpha}{N-2}$, and $\frac{N}{2}(p+q)-N-\alpha=2$.
\end{remark}

This paper is motivated by Bieganowski, d'Avenia and Schino \cite{Bieganowski-dAvenia-Schino 2024} where the Schr\"{o}dinger equation
\begin{equation}\label{e2.2}
\begin{cases}
-\Delta u+\lambda u=f(u),\quad \text{in}\  \mathbb{R}^N, \\
\int_{\mathbb{R}^N}|u|^2dx=\rho^2
\end{cases}
\end{equation}
is considered under the following conditions:\\
(f1) $f$ is continuous and $f(t)\lesssim |t|+|t|^{2^*-1}$;\\
(f2) $\lim_{t\to 0}\frac{F(t)}{t^2}=0$;\\
(f3) $\lim_{t\to 0}\frac{F(t)}{t^{2+4/N}}=+\infty$.\\
Following the arguments there, for $a,R >0$, we define
\begin{equation*}
\mathcal{D}_{a}:=\left\{u\in
H^1(\mathbb{R}^N):\|u\|_2\leq a\right\}\ \text{and}\ \mathcal{U}_{R}(a):=\left\{u\in
\mathcal{D}_{a}:\|\nabla u\|_2<R\right\},
\end{equation*}
and consider the minimizing problem
\begin{equation*}
m_{R}(a):=\inf_{u\in \mathcal{U}_{R}(a)} E(u).
\end{equation*}
By using the concentration compactness principle and the geometric characteristics of the functional, we first show that $m_{R_0}(\rho)$ is attained (see Lemma \ref{lem3.2} for the definition of $R_0$). Then using some kind of `monotonicity', we can show that such a minimizer must be in $\mathcal{S}_\rho$ and be a normalized solution to (\ref{e1.3}). Finally, using the  geometric characteristics of $E(u_\tau)$,  we show that  $m_{R_0}(\rho)$ equals to the ground state energy.

Compared with \cite{Bieganowski-dAvenia-Schino 2024},  there are some differences and difficulties in our proofs:

(1) As to the simplest case $b=0$, for the interaction of $I_{\alpha}\ast F(u)$ and $F(u)$, we make some skill to obtain the lower bound of $E(u)$ (Lemma \ref{lem3.1}) and choose a different scaling to obtain the subadditivity property (Lemma \ref{lem3.4}).

(2) As to the Hardy-Littlewood-Sobolev lower critical case $b=1$ which is also the most difficult case, we use the extremal function of $S_2$ (Lemma \ref{lem2.7}) to obtain the upper bound of $m_{R_0}(\rho)$ (Lemma \ref{lem3.3}) and use the estimate (\ref{e3.8}) to exclude the vanishing case in the concentration compactness principle (Lemma \ref{lem3.6}). In the proof of Theorem \ref{thm1.2}, to compare the values of $E(u_\tau)$ among $\tau=1$, $\tau\ll 1$ and $\tau\in (\tau_1,\tau_2)$ with $\tau_2<1$, different lower bound and upper bound estimates of $E(u_\tau)$ are made and then by which $m_{R_0}(\rho)$ is proved to be the ground state energy.

(3) Last but not least, by using some `monotonicity' (Remark \ref{rmk3.2} (2)), we can show that the minimizer of $m_{R_0}(\rho)$ must be in $\mathcal{S}_\rho$ and then show that $m_{R_0}(\rho)=\inf\{E(w):\ w\in \mathcal{S}_{\rho},\
E|_{\mathcal{S}_{\rho}}'(w)=0\}$, which is different from the proof of Theorem 1.2 in  \cite{Bieganowski-dAvenia-Schino 2024} and Proposition 1.4 there.

\medskip

\textbf{Organization of the paper}.  In Section 2, we give some preliminary results used in this paper. Section 3 is devoted to the proofs of Theorems \ref{thm1.1} and \ref{thm1.2}.

\medskip

\textbf{Basic notations}:  The usual norm of $u\in
L^p(\mathbb{R}^N)$ is denoted by $\|u\|_p$. $D^{1,2}(\mathbb{R}^N)$ denotes the completion of $C^{\infty}_0 (\mathbb{R}^N)$ with respect to the norm $\|\nabla \cdot \|_2$. \ $o_n(1)$ denotes a real
sequence with $o_n(1) \to 0$ as $n\to +\infty$. `$\rightarrow$'
denotes  strong convergence and `$\rightharpoonup$' denotes weak
convergence. The notation
$A\lesssim B$ means that $A\leq CB$ for some constant $C> 0$. For $N\geq 3$, $2^*:=\frac{2N}{N-2}$. $B(y,1):=\{x\in \mathbb{R}^N:|x-y|\leq 1\}$.

\section{Preliminaries}
\setcounter{section}{2} \setcounter{equation}{0}

The following well-known Hardy-Littlewood-Sobolev inequalities and Sobolev inequality can be
found in \cite{Lieb-Loss 2001}.

\begin{lemma}\label{lem HLS}
Let $N\geq 1$, $p$, $r>1$ and $0<\alpha<N$ with
$1/p+(N-\alpha)/N+1/r=2$. Let $u\in L^p(\mathbb{R}^N)$ and $v\in
L^r(\mathbb{R}^N)$. Then there exists a sharp constant
$C(N,\alpha,p)$, independent of $u$ and $v$, such that
\begin{equation*}
\left|\int_{\mathbb{R}^N}\int_{\mathbb{R}^N}\frac{u(x)v(y)}{|x-y|^{N-\alpha}}dxdy\right|\leq
C(N,\alpha,p)\|u\|_p\|v\|_r.
\end{equation*}
If $p=r=\frac{2N}{N+\alpha}$, then
\begin{equation*}
C(N,\alpha,p)=C_\alpha(N):=\pi^{\frac{N-\alpha}{2}}\frac{\Gamma(\frac{\alpha}{2})}{\Gamma(\frac{N+\alpha}{2})}\left\{\frac{\Gamma(\frac{N}{2})}{\Gamma(N)}\right\}^{-\frac{\alpha}{N}}.
\end{equation*}
\end{lemma}

\begin{lemma}\label{lem2.1}
 Let $N\geq 3$, $\alpha\in (0,N)$, and
\begin{equation}\label{e3.14}
\begin{split}
S_1:&=\inf_{ u\in
D^{1,2}(\mathbb{R}^N)\setminus\{0\}}\frac{\int_{\mathbb{R}^N}|\nabla
u|^2dx }{\left(\int_{\mathbb{R}^N}\left(I_{\alpha}\ast
|u|^{\frac{N+\alpha}{N-2}}\right)|u|^{\frac{N+\alpha}{N-2}}dx\right)^{\frac{N-2}{N+\alpha}}}.
\end{split}
\end{equation}
Then $S_1>0$.
\end{lemma}

\begin{lemma}\label{lem2.7}
Let $N\geq 1$, $\alpha\in (0,N)$, and
\begin{equation}\label{e1.2}
S_2:=\inf_{u\in
H^1(\mathbb{R}^N)}\frac{\int_{\mathbb{R}^N}|u|^2dx}
{\left(\int_{\mathbb{R}^N}(I_\alpha\ast|u|^{\frac{N+\alpha}{N}})|u|^{\frac{N+\alpha}{N}}dx\right)^{\frac{N}{N+\alpha}}}.
\end{equation}
Then $S_2>0$ and  the minimizer of $S_2$ is
$$u(x)=b\left(\frac{\delta}{\delta^2+|x-y|^2}\right)^{N/2},\ b\in
\mathbb{R},\ \delta>0,\ y\in \mathbb{R}^N.$$
\end{lemma}

\begin{lemma}\label{lem22.1}
Let $N\geq 3$, and
 \begin{equation*}
 S_3:=\inf_{u\in D^{1,2}(\mathbb{R}^N)\backslash\{0\}}\frac{\|\nabla u\|_2^2}{\|u\|_{2^*}^2}.
\end{equation*}
Then $S_3>0$.
\end{lemma}

\section{Proofs of the main results}
\setcounter{section}{3} \setcounter{equation}{0}

From  ($G_1$), There exists $C_0>0$ such that
\begin{equation}\label{e3.13}
|G(t)|\leq C_0(|t|^{\frac{N+\alpha}{N}}+|t|^{\frac{N+\alpha}{N-2}}),\ \forall\ t\in\mathbb{R}.
\end{equation}

We first obtain a lower bound of $E(u)$ on $\mathcal{D}_{\rho}$.

\begin{lemma}\label{lem3.1}
Let ($G_1$)  hold, $S_1$, $S_2$, $S_3$, $C_\alpha(N)$ be defined in Lemmas \ref{lem HLS}-\ref{lem22.1}, and $C_0$ be as in (\ref{e3.13}). Then
\begin{equation*}
E(u)\geq -\frac{1}{2}b^2S_2^{-\frac{N+\alpha}{N}}\rho^{2\frac{N+\alpha}{N}}+h(\rho,\|\nabla u\|_2)\|\nabla u\|_2^2,\ \forall u\in \mathcal{D}_{\rho},
\end{equation*}
where $h:(0,+\infty)\times (0,+\infty)\to \mathbb{R}$ is defined by
\begin{equation*}
h(a,t):=\frac{1}{2}-C_1a^{2\frac{N+\alpha}{N}}t^{-2}-C_2t^{2\frac{N+\alpha}{N-2}-2},
\end{equation*}
\begin{equation*}
C_1:=\frac{1}{2}C_0(2b+C_0)S_2^{-\frac{N+\alpha}{N}}+\frac{1}{2}C_0(b+C_0)A_{\alpha}(N)C_\alpha(N)S_3^{-\frac{1}{2}\frac{N+\alpha}{N-2}},
\end{equation*}
and
\begin{equation*}
C_2:=\frac{1}{2}C_0^2S_1^{-\frac{N+\alpha}{N-2}}+\frac{1}{2}C_0(b+C_0)A_{\alpha}(N)C_\alpha(N)S_3^{-\frac{1}{2}\frac{N+\alpha}{N-2}}.
\end{equation*}
\end{lemma}

\begin{proof}
By using Lemmas \ref{lem HLS}-\ref{lem22.1}, we obtain that
\begin{equation*}
\int_{\mathbb{R}^N}(I_\alpha\ast|u|^{\frac{N+\alpha}{N-2}})|u|^{\frac{N+\alpha}{N-2}}dx\leq S_1^{-\frac{N+\alpha}{N-2}}\|\nabla u\|_2^{2\frac{N+\alpha}{N-2}},
\end{equation*}
\begin{equation*}
\int_{\mathbb{R}^N}(I_\alpha\ast|u|^{\frac{N+\alpha}{N}})|u|^{\frac{N+\alpha}{N}}dx\leq S_2^{-\frac{N+\alpha}{N}}\| u\|_2^{2\frac{N+\alpha}{N}},
\end{equation*}
\begin{equation*}
\begin{split}
\int_{\mathbb{R}^N}(I_\alpha\ast|u|^{\frac{N+\alpha}{N}})|u|^{\frac{N+\alpha}{N-2}}dx&\leq A_{\alpha}(N)C_{\alpha}(N)\|u\|_2^{\frac{N+\alpha}{N}}\| u\|_{2^*}^{\frac{N+\alpha}{N-2}}\\
&\leq A_{\alpha}(N)C_{\alpha}(N)S_3^{-\frac{1}{2}\frac{N+\alpha}{N-2}}\|u\|_2^{\frac{N+\alpha}{N}}\| \nabla u\|_{2}^{\frac{N+\alpha}{N-2}},
\end{split}
\end{equation*}
which combined with   (\ref{e3.13})  gives that
\begin{equation}\label{e3.11}
\begin{split}
E(u)&\geq \frac{1}{2}\|\nabla u\|_2^2-\frac{1}{2}\int_{\mathbb{R}^N}[I_\alpha\ast(b|u|^{\frac{N+\alpha}{N}}+C_0|u|^{\frac{N+\alpha}{N}}+C_0|u|^{\frac{N+\alpha}{N-2}})]\\
&\qquad\qquad\qquad\qquad\qquad(b|u|^{\frac{N+\alpha}{N}}+C_0|u|^{\frac{N+\alpha}{N}}+C_0|u|^{\frac{N+\alpha}{N-2}})dx\\
&\geq \frac{1}{2}\|\nabla u\|_2^2-\frac{1}{2}b^2S_2^{-\frac{N+\alpha}{N}}\| u\|_2^{2\frac{N+\alpha}{N}}-\frac{1}{2}C_0(2b+C_0)S_2^{-\frac{N+\alpha}{N}}\| u\|_2^{2\frac{N+\alpha}{N}}\\
&\qquad -C_0(b+C_0)A_{\alpha}(N)C_{\alpha}(N)S_3^{-\frac{1}{2}\frac{N+\alpha}{N-2}}\|u\|_2^{\frac{N+\alpha}{N}}\| \nabla u\|_{2}^{\frac{N+\alpha}{N-2}}\\
&\qquad -\frac{1}{2}C_0^2S_1^{-\frac{N+\alpha}{N-2}}\|\nabla u\|_2^{2\frac{N+\alpha}{N-2}}.
\end{split}
\end{equation}
By using $ab\leq \frac{1}{2}(a^2+b^2)$, we  further obtain that
\begin{equation}\label{e3.2}
E(u)\geq -\frac{1}{2}b^2S_2^{-\frac{N+\alpha}{N}}\| u\|_2^{2\frac{N+\alpha}{N}}+\frac{1}{2}\|\nabla u\|_2^2-C_1
\|u\|_2^{2\frac{N+\alpha}{N}}-C_2\| \nabla u\|_{2}^{2\frac{N+\alpha}{N-2}},
\end{equation}
which implies the lemma for any $u\in \mathcal{D}_{\rho}$.
\end{proof}

Now we give some properties of $h(a,t)$.

\begin{lemma}\label{lem3.2}
The following facts hold.\\
(g1) For every $a>0$, the function $t\mapsto h(a, t)$ has a unique critical point, which is a global maximizer.\\
(g2) If
\begin{equation}\label{e3.1}
\rho^{2\frac{N+\alpha}{N}}<\left(\frac{N-2}{2C_2(N+\alpha)}\right)^{\frac{N+\alpha}{2+\alpha}}\frac{C_2}{C_1}{\frac{2+\alpha}{N-2}}
\end{equation}
holds, then there exist $R_0,R_1 > 0$, $R_0 < R_1$, such that $h(\rho,R_0)=h(\rho,R_1)=0$, $h(\rho, t)>0$
for $t\in(R_0,R_1)$, and $h(\rho, t)<0$ for $t\in(0,R_0)\cup(R_1,+\infty)$.\\
(g3) If $t>0$ and $a_1\geq a_2>0$, then for every $s\in [ta_2/a_1, t]$ there holds  $h(a_2, s)\geq h(a_1, t)$.
\end{lemma}

\begin{proof}
For fixed $a>0$,  set $w_{a}(t):=h(a,t)$. The unique critical point of $w_{a}(t)$ is $$t_0=\left(\frac{C_1(N-2)a^{2\frac{N+\alpha}{N}}}{C_2(2+\alpha)}\right)^{\frac{N-2}{2(N+\alpha)}}$$
and the maximum is
\begin{equation*}
w_{a}(t_0)=\frac{1}{2}-C_2{\frac{N+\alpha}{N-2}}\left(\frac{C_1(N-2)a^{2\frac{N+\alpha}{N}}}{C_2(2+\alpha)}\right)^{\frac{2+\alpha}{N+\alpha}}.
\end{equation*}
Thus if (\ref{e3.1}) holds, we obtain that  $w_{\rho}(t_0)>0$. This proves (g1) and (g2).

Concerning (g3), it is clear that $h(a_2, t)\geq h(a_1, t)$ for all $t > 0$. Moreover,
\begin{equation*}
\begin{split}
h(a_2,t\frac{a_2}{a_1})-h(a_1,t)&=C_1 \left[1-\left(\frac{a_2}{a_1}\right)^{\frac{2\alpha}{N}}\right]a_1^{2\frac{N+\alpha}{N}}t^{-2}\\
&\qquad+C_2 \left[1-\left(\frac{a_2}{a_1}\right)^{2\frac{N+\alpha}{N-2}-2}\right]t^{2\frac{N+\alpha}{N-2}-2}\\
&\geq 0.
\end{split}
\end{equation*}
Hence we obtain (g3) thanks to (g1).
\end{proof}

Let us obtain the upper bound of $m_R(a)$.

\begin{lemma}\label{lem3.3}
Let ($G_1$) and ($G_2$) hold, then $$m_R(a)\in \left(-\infty,-\frac{1}{2}b^2S_2^{-\frac{N+\alpha}{N}}a^{2\frac{N+\alpha}{N}}\right)$$ for every $a,R>0$.
\end{lemma}

\begin{proof}
Firstly, it follows from (\ref{e3.2}) that $m_R(a)>-\infty$. Then, by Lemma \ref{lem2.7}, we choose $v\in H^1(\mathbb{R}^N)$ such that
\begin{equation*}
S_2=\frac{\int_{\mathbb{R}^N}|v|^2dx}
{\left(\int_{\mathbb{R}^N}(I_\alpha\ast|v|^{\frac{N+\alpha}{N}})|v|^{\frac{N+\alpha}{N}}dx\right)^{\frac{N}{N+\alpha}}}.
\end{equation*}
 Let $u:=\frac{av}{\|v\|_2}$ and $u_{\tau}(x):=\tau^{\frac{N}{2}}u(\tau x)$ for $\tau>0$, then $\|u_\tau\|_2=a$ for all $\tau>0$, and
\begin{equation*}
\begin{split}
E(u_\tau)&=-\frac{1}{2}b^2S_2^{-\frac{N+\alpha}{N}}a^{2\frac{N+\alpha}{N}}\\
&\qquad +\frac{1}{2}\tau^2\left\{\|\nabla u\|_2^2-\frac{1}{2}\frac{1}{\tau^{N+\alpha+2}}\int_{\mathbb{R}^N}(I_\alpha\ast G(\tau^{\frac{N}{2}}u))G(\tau^{\frac{N}{2}}u)dx\right.\\
&\qquad\qquad\qquad\qquad\left.-b\frac{1}{\tau^{(N+\alpha)/2+2}}\int_{\mathbb{R}^N}(I_\alpha\ast |u|^{\frac{N+\alpha}{N}})G(\tau^{\frac{N}{2}}u)dx
\right\}.
\end{split}
\end{equation*}
Note that, from ($G_2$), $G(\tau^{\frac{N}{2}}u)> 0$  for sufficiently small $\tau> 0$. Therefore, from Fatou's
Lemma and again ($G_2$),
\begin{equation*}
\lim_{\tau\to 0^+}\frac{1}{\tau^{N+\alpha+2}}\int_{\mathbb{R}^N}(I_\alpha\ast G(\tau^{\frac{N}{2}}u))G(\tau^{\frac{N}{2}}u)dx=+\infty\ \text{if}\ b=0
\end{equation*}
and
\begin{equation*}
\lim_{\tau\to 0^+}\frac{1}{\tau^{(N+\alpha)/2+2}}\int_{\mathbb{R}^N}(I_\alpha\ast |u|^{\frac{N+\alpha}{N}})G(\tau^{\frac{N}{2}}u)dx=+\infty\  \text{if}\ b=1.
\end{equation*}
This implies that $E(u_\tau) < -\frac{1}{2}b^2S_2^{-\frac{N+\alpha}{N}}a^{2\frac{N+\alpha}{N}}$ for sufficiently small $\tau> 0$, and since $u_\tau\in \mathcal{U}_R(a)$ provided $\tau > 0$ is
small, we complete the proof.
\end{proof}

Now we show a property ensuring that $m_{R_0}(\rho)$ cannot be attained at the boundary of $\mathcal{U}_{R_0}(\rho)$ under the conditions ($G_1$), ($G_2$) and (\ref{e3.1}).

\begin{remark}\label{rmk3.1}
By Lemma \ref{lem3.3}, $$m_{R_0}(\rho)+\frac{1}{2}b^2S_2^{-\frac{N+\alpha}{N}}\rho^{2\frac{N+\alpha}{N}}<0,$$ and by Lemma \ref{lem3.2} (g2), there exists $\epsilon>0$ such that
\begin{equation*}
0\geq h(\rho, t)>\frac{m_{R_0}(\rho)+\frac{1}{2}b^2S_2^{-\frac{N+\alpha}{N}}\rho^{2\frac{N+\alpha}{N}}}{R_0^2},\ \forall\ t\in [R_0-\epsilon,R_0],
\end{equation*}
which combined with Lemma \ref{lem3.1} gives that for all $u\in \mathcal{D}_{\rho}$ such that $R_0-\epsilon\leq \|\nabla u\|_2\leq R_0$ there holds
\begin{equation*}
\begin{split}
E(u)&\geq -\frac{1}{2}b^2S_2^{-\frac{N+\alpha}{N}}\rho^{2\frac{N+\alpha}{N}}+h(\rho, \|\nabla u\|_2)\|\nabla u\|_2^2\\
&> -\frac{1}{2}b^2S_2^{-\frac{N+\alpha}{N}}\rho^{2\frac{N+\alpha}{N}}+\frac{m_{R_0}(\rho)+\frac{1}{2}b^2S_2^{-\frac{N+\alpha}{N}}\rho^{2\frac{N+\alpha}{N}}}{R_0^2}\|\nabla u\|_2^2\\
&\geq -\frac{1}{2}b^2S_2^{-\frac{N+\alpha}{N}}\rho^{2\frac{N+\alpha}{N}}+\frac{m_{R_0}(\rho)+\frac{1}{2}b^2S_2^{-\frac{N+\alpha}{N}}\rho^{2\frac{N+\alpha}{N}}}{R_0^2}{R_0^2}\\
&=m_{R_0}(\rho).
\end{split}
\end{equation*}
\end{remark}

Next, we show the following subadditivity property.

\begin{lemma}\label{lem3.4}
If ($G_1$), ($G_2$), and (\ref{e3.1}) hold, then for every $a\in(0, \rho)$ there holds
\begin{equation*}
m_{R_0}(\rho)\leq  m_{R_0}(a)+m_{R_0}(\sqrt{\rho^2-a^2}).
\end{equation*}
Moreover, the inequality above is strict if at least one between $m_{R_0}(\sqrt{\rho^2-a^2})$  and $m_{R_0}(a)$
is attained.
\end{lemma}

\begin{proof}
We claim that
\begin{equation}\label{e3.3}
 m_{R_0}(\theta a)\leq  \theta^2m_{R_0}(a)
\end{equation}
for every $\theta\in [1, \rho/a]$ and that the inequality is strict if $m_{R_0}(a)$ is attained and $\theta > 1$.

By (\ref{e3.2}), Lemma \ref{lem3.3}, and the definition of $m_{R_0}(a)$, for any fixed
$$\epsilon\in \left(0,-\frac{1}{2}b^2S_2^{-\frac{N+\alpha}{N}}a^{2\frac{N+\alpha}{N}}-m_{R_0}(a)\right),$$
there exists
$u\in \mathcal{U}_{R_0}(a)$ such that
\begin{equation*}
\begin{split}
-\frac{1}{2}b^2S_2^{-\frac{N+\alpha}{N}}a^{2\frac{N+\alpha}{N}}+&h(a,\|\nabla u\|_2)\|\nabla u\|_2^2\\
&\qquad\leq E(u)\leq m_{R_0}(a)+\epsilon<-\frac{1}{2}b^2S_2^{-\frac{N+\alpha}{N}}a^{2\frac{N+\alpha}{N}},
\end{split}
\end{equation*}
which implies that
\begin{equation}\label{e3.4}
h(a,\|\nabla u\|_2)<0.
\end{equation}
If $\|\nabla u\|_2\geq {R_0a}/{\rho}$, then from Lemma \ref{lem3.2} (g3) with $a_1 =\rho$, $a_2 = a$, $t = R_0$, and $s = \|\nabla u\|_2$, we obtain
$h(a,\|\nabla u\|_2)\geq h(\rho,R_0)=0$, which contradicts  (\ref{e3.4}). Therefore $\|\nabla u\|_2< {R_0a}/{\rho}$.
Define $v:=u(\theta^{-\frac{2}{N+\alpha}}x)$. Direct calculations give that
\begin{equation*}
\|v\|_2=\theta^{\frac{N}{N+\alpha}}\|u\|_2\leq  \theta\|u\|_2\leq \theta a,
\end{equation*}
\begin{equation*}
\|\nabla v\|_2=\theta^{\frac{N-2}{N+\alpha}}\|\nabla u\|_2\leq  \theta \frac{R_0a}{\rho}\leq R_0,
\end{equation*}
and
\begin{equation*}
E(v)=\theta^2\left(\frac{1}{2}\theta^{-\frac{2(2+\alpha)}{N+\alpha}}\|\nabla  u\|_2^2-\frac{1}{2}\int_{\mathbb{R}^N}(I_\alpha\ast F(u))F(u)dx\right).
\end{equation*}
Then $v\in \mathcal{U}_{R_0}(\theta a)$ and
\begin{equation*}
m_{R_0}(\theta  a)\leq E(v)\leq \theta^2E(u)\leq  \theta^2(m_{R_0}(a)+\epsilon),
\end{equation*}
which implies (\ref{e3.3}) by the arbitrariness of $\epsilon$.  In addition, if $m_{R_0}(a)$ is attained and $\theta > 1$, we can take $u\in \mathcal{U}_{R_0}(a)$ as the minimizer and, repeating
the previous argument, we have
\begin{equation}\label{e3.5}
m_{R_0}(\theta  a)\leq E(v)< \theta^2 E(u)=\theta^2m_{R_0}(a).
\end{equation}
This proves the claim.

If $a^2>\rho^2-a^2$, we obtain
\begin{equation*}
m_{R_0}(\rho)\stackrel{(1)}{\leq}\frac{\rho^2}{a^2}m_{R_0}(a)=m_{R_0}(a)+\frac{\rho^2-a^2}{a^2}m_{R_0}(a)\stackrel{(2)}{\leq}m_{R_0}(a)+m_{R_0}(\sqrt{\rho^2-a^2}),
\end{equation*}
where the inequality (1) (respectively, (2)) is strict if $m_{R_0}(a)$ (respectively, $m_{R_0}(\sqrt{\rho^2-a^2})$) is attained.
Analogously, if $a^2<\rho^2-a^2$, we obtain
\begin{equation*}
\begin{split}
m_{R_0}(\rho)&\stackrel{(3)}{\leq}\frac{\rho^2}{\rho^2-a^2}m_{R_0}(\sqrt{\rho^2-a^2})
=m_{R_0}(\sqrt{\rho^2-a^2})+\frac{a^2}{\rho^2-a^2}m_{R_0}(\sqrt{\rho^2-a^2})\\
&\stackrel{(4)}{\leq}m_{R_0}(\sqrt{\rho^2-a^2})+m_{R_0}(a),
\end{split}
\end{equation*}
where the inequality (3) (respectively, (4)) is strict if $m_{R_0}(\sqrt{\rho^2-a^2})$
(respectively, $m_{R_0}(a)$) is attained.
Finally, if $a^2=\rho^2-a^2$ (i.e. $\sqrt{2}a = \rho$), then
\begin{equation*}
m_{R_0}(\rho)=m_{R_0}(\sqrt{2}a)\stackrel{(5)}{\leq}  2m_{R_0}(a)
\end{equation*}
follows from (\ref{e3.3}). And by (\ref{e3.5}) the inequality (5) is strict if $m_{R_0}(a)$ is attained.
\end{proof}

We obtain the monotonicity of $m_R(a)$.

\begin{remark}\label{rmk3.2}
(1) By the definition of $m_R(a)$ and $\mathcal{U}_R(a)$, we obtain that
\begin{equation*}
m_{R}(a)\leq m_{R}(b),\ \forall a>b>0, R>0.
\end{equation*}
(2) By Lemmas \ref{lem3.3} and \ref{lem3.4}, if ($G_1$), ($G_2$), and (\ref{e3.1}) hold, then
\begin{equation*}
m_{R_0}(\rho)< m_{R_0}(a),\ \forall a\in (0,\rho).
\end{equation*}
\end{remark}

The next lemma shows the relative compactness in $L^p(\mathbb{R}^N)$ of minimizing sequences.

\begin{lemma}\label{lem3.6}
Assume that ($G_1$)-($G_3$) and (\ref{e3.1}) hold. If $\rho_n\to \rho$, $R_n\to R_0$, and $\{u_n\}\subset H^1(\mathbb{R}^N)$ is such
that $u_n\in \mathcal{U}_{R_n}(\rho_n)$ for every $n\in \mathbb{N}$ and $E(u_n)\to m_{R_0}(\rho)$, then there exist $\{y_n\}\subset \mathbb{R}^N$ and $\bar{u}\in \mathcal{D}_{\rho}\backslash\{0\}$  such that $\|\nabla \bar{u}\|_2\leq R_0$ and $u_n(\cdot+y_n)\to \bar{u}$ in $L^p(\mathbb{R}^N)$ for every $p\in [2,2^*)$.
\end{lemma}

\begin{proof}
First of all, since $\{u_n\}$ is bounded in $H^1(\mathbb{R}^N)$, it suffices to prove the statement for $p=2$ in virtue
of the interpolation inequality and the Sobolev embedding.

Assume by contradiction that
\begin{equation*}
\lim_{n\to  +\infty}\sup_{y\in \mathbb{R}^N}\int_{B(y,1)}|u_n|^2dx=0.
\end{equation*}
Then, from Lions' lemma (\cite{Lions1984}, Lemma I.1), we obtain that
\begin{equation}\label{e3.6}
\lim_{n\to  +\infty} \|u_n\|_q=0, \ \forall q\in (2,2^*).
\end{equation}
Fix $r\in (\frac{N+\alpha}{N},\frac{N+\alpha}{N-2})$ and $\epsilon>0$.  From (\ref{e3.13}) and ($G_3$), there exists $c=c(r,\epsilon,C_0)>0$ such that for every $t\in \mathbb{R}$,
\begin{equation}\label{e3.7}
G(t)\leq \epsilon|t|^{\frac{N+\alpha}{N}}+c|t|^r+C_0|t|^{\frac{N+\alpha}{N-2}}.
\end{equation}
Choose $\eta>0$ small enough such that
\begin{equation*}
\frac{1}{p_1}:=\frac{N+\alpha}{2N}-\eta<1,\ \frac{1}{q_1}:=\frac{N+\alpha}{2N}+\eta<1,
\end{equation*}
\begin{equation*}
\frac{N+\alpha}{N}p_1\in (2,2^*),\ \frac{N+\alpha}{N-2}q_1\in (2,2^*),
\end{equation*}
then $\frac{1}{p_1}+\frac{N-\alpha}{N}+\frac{1}{q_1}=2$, and thus by Lemma \ref{lem HLS} and (\ref{e3.6}), we obtain
\begin{equation}\label{e3.8}
\begin{split}
\int_{\mathbb{R}^N}(I_\alpha\ast|u_n|^{\frac{N+\alpha}{N}})|u_n|^{\frac{N+\alpha}{N-2}}dx&\leq \||u_n|^{\frac{N+\alpha}{N}}\|_{p_1}\||u_n|^{\frac{N+\alpha}{N-2}}\|_{q_1}\\
&=\|u_n\|_{\frac{N+\alpha}{N}p_1}^{\frac{N+\alpha}{N}}\|u_n\|_{\frac{N+\alpha}{N-2}q_1}^{\frac{N+\alpha}{N-2}}=o_n(1).
\end{split}
\end{equation}
Using Lemma \ref{lem HLS} and   (\ref{e3.6})-(\ref{e3.8}), similarly  to (\ref{e3.11}), there exists a constant $C>0$ independent of  $\epsilon$  and $n$ such that
\begin{equation*}
\begin{split}
E(u_n)&\geq -\frac{1}{2}b^2S_2^{-\frac{N+\alpha}{N}}\| u_n\|_2^{2\frac{N+\alpha}{N}}+\frac{1}{2}\|\nabla u_n\|_2^2 \\
&\qquad-\frac{1}{2}C_0^2S_1^{-\frac{N+\alpha}{N-2}}\|\nabla u_n\|_2^{2\frac{N+\alpha}{N-2}}-C\epsilon+o_n(1)\\
&\geq -\frac{1}{2}b^2S_2^{-\frac{N+\alpha}{N}}\|u_n\|_2^{2\frac{N+\alpha}{N}}+\|\nabla u_n\|_2^2\left(\frac{1}{2}-C_2\|\nabla u_n\|_2^{2\frac{N+\alpha}{N-2}-2}\right)-C\epsilon+o_n(1)\\
&\geq -\frac{1}{2}b^2S_2^{-\frac{N+\alpha}{N}}\rho_n^{2\frac{N+\alpha}{N}}+\|\nabla u_n\|_2^2\left(\frac{1}{2}-C_2R_n^{2\frac{N+\alpha}{N-2}-2}\right)-C\epsilon+o_n(1).
\end{split}
\end{equation*}
This combined with $h(\rho,R_0)=0$ gives that
\begin{equation*}
\begin{split}
m_{R_0}(\rho)&=\lim_{n\to+\infty}E(u_n)\\
&\geq -\frac{1}{2}b^2S_2^{-\frac{N+\alpha}{N}}\rho^{2\frac{N+\alpha}{N}}+\left(\frac{1}{2}-C_2R_0^{2\frac{N+\alpha}{N-2}-2}\right)\liminf_{n\to+\infty}\|\nabla u_n\|_2^2-C\epsilon\\
&= -\frac{1}{2}b^2S_2^{-\frac{N+\alpha}{N}}\rho^{2\frac{N+\alpha}{N}}+C_1\rho^{2\frac{N+\alpha}{N}}R_0^{-2}\liminf_{n\to+\infty}\|\nabla u_n\|_2^2-C\epsilon\\
&\geq -\frac{1}{2}b^2S_2^{-\frac{N+\alpha}{N}}\rho^{2\frac{N+\alpha}{N}}-C\epsilon,
\end{split}
\end{equation*}
which implies $m_{R_0}(\rho)\geq  -\frac{1}{2}b^2S_2^{-\frac{N+\alpha}{N}}\rho^{2\frac{N+\alpha}{N}}$ by the arbitrariness of $\epsilon$. That contradicts Lemma \ref{lem3.3}. Hence there exists $\{y_n\}\subset \mathbb{R}^N$ such that
\begin{equation*}
\limsup_{n\to+\infty}\int_{B(0,1)}|u_n(x+y_n)|^2dx>0.
\end{equation*}

As a consequence, there exists $\bar{u}\in \mathcal{D}_\rho\backslash\{0\}$ with $\|\nabla \bar{u}\|_2\leq R_0$ such that, setting $v_n(x):= u_n(x+y_n)-\bar{u}(x)$,
up to a subsequence $v_n\rightharpoonup0$   in $H^1(\mathbb{R}^N)$ and $v_n(x)\to 0$ for a.e. $x\in \mathbb{R}^N$.

In the following, we just need to prove $\|v_n\|_2\to 0$ as $n\to +\infty$. Assume by contradiction that $\beta_1:=\lim_{n\to+\infty}\|v_n\|_2>0$ along a subsequence. Denote $\beta_2:=\|\bar{u}\|_2$.  From the Brezis-Lieb Lemma
and the weak convergence in $H^1(\mathbb{R}^N)$ we get
\begin{equation*}
\begin{split}
\lim_{n\to+\infty}(E(u_n)-E(v_n))&=E(\bar{u}),\\
\lim_{n\to+\infty}(\|u_n\|_2^2-\|v_n\|_2^2) &=\|\bar{u}\|_2^2,\\
 \lim_{n\to+\infty}(\|\nabla u_n\|_2^2-\|\nabla v_n\|_2^2) &=\|\nabla \bar{u}\|_2^2,
\end{split}
\end{equation*}
and so, in particular, $\beta_2^2\leq \rho^2-\beta_1^2$ and $\|\nabla v_n\|_2<R_0$ for $n\gg1$. Now, for every $n$ define
\begin{equation*}
\tilde{v}_n:=\left\{
\begin{array}{ll}
v_n,&\text{if}\ \|v_n\|_2\leq \beta_1,\\
\frac{\beta_1}{\|v_n\|_2}v_n,&\text{if}\ \|v_n\|_2> \beta_1.
\end{array}
\right.
\end{equation*}
Then $\tilde{v}_n\in \mathcal{U}_{R_0}(\beta_1)$. Set $a_n:=\frac{\beta_1}{\|v_n\|_2}$. Direct calculations give that
\begin{equation*}
\begin{split}
E(\tilde{v}_n)-E(v_n)&=\frac{1}{2}\|\nabla \tilde{v}_n\|_2^2-\frac{1}{2}\|\nabla v_n\|_2^2\\
&\qquad+\frac{1}{2}\int_{\mathbb{R}^N}(I_\alpha\ast F(v_n))F(v_n)dx-\frac{1}{2}\int_{\mathbb{R}^N}(I_\alpha\ast F(\tilde{v}_n))F(\tilde{v}_n)dx\\
&=\frac{1}{2}\|\nabla \tilde{v}_n\|_2^2-\frac{1}{2}\|\nabla v_n\|_2^2\\
&\qquad+\frac{1}{2}\int_{\mathbb{R}^N}(I_\alpha\ast (F(v_n)+F(\tilde{v}_n)))(F(v_n)-F(\tilde{v}_n))dx
\end{split}
\end{equation*}
 and
 \begin{equation*}
\begin{split}
&\int_{\mathbb{R}^N}(I_\alpha\ast (F(v_n)+F(\tilde{v}_n)))(F(v_n)-F(\tilde{v}_n))dx\\
&\quad=\int_{\mathbb{R}^N}(I_\alpha\ast (F(v_n)+F(\tilde{v}_n)))\left(\int_{0}^{1}f(\tilde{v}_n+\theta(v_n-\tilde{v}_n))(v_n-\tilde{v}_n)d\theta\right)dx\\
&\quad=(1-a_n)\int_{\mathbb{R}^N}(I_\alpha\ast (F(v_n)+F(\tilde{v}_n)))\left(\int_{0}^{1}f(\tilde{v}_n+\theta(v_n-\tilde{v}_n))v_nd\theta\right)dx.
\end{split}
\end{equation*}
Since $|\tilde{v}_n+\theta(v_n-\tilde{v}_n)|\leq 4|v_n|$,
using ($G_1$) and $a_n\to 1$ as $n\to+\infty$, we obtain that
\begin{equation}\label{e3.9}
\lim_{n\to+\infty}(E(v_n)-E(\tilde{v}_n))=0.
\end{equation}
Hence, by Lemma \ref{lem3.4}, Remark \ref{rmk3.2} (1), (\ref{e3.9}) and the definition of $m_R(a)$, we obtain that
\begin{equation}\label{e3.10}
\begin{split}
m_{R_0}(\rho)&=\lim_{n\to+\infty}E(u_n)=E(\bar{u})+\lim_{n\to+\infty}E(v_n)=E(\bar{u})+\lim_{n\to+\infty}E(\tilde{v}_n)\\
&\geq  m_{R_0}(\beta_2)+m_{R_0}(\beta_1)\geq m_{R_0}(\beta_2)+m_{R_0}(\sqrt{\rho^2-\beta_2^2})\geq  m_{R_0}(\rho).
\end{split}
\end{equation}
This implies that $E(\bar{u})=m_{R_0}(\beta_2)$ and so, again from Lemma \ref{lem3.4},
\begin{equation*}
m_{R_0}(\beta_2)+m_{R_0}(\sqrt{\rho^2-\beta_2^2})> m_{R_0}(\rho),
\end{equation*}
which contradicts (\ref{e3.10}). So we must have $\|v_n\|_2\to 0$ as $n\to +\infty$ and  the proof is complete.
\end{proof}

Now we  prove  the following result.

\begin{theorem}\label{thm3.1}
Assume that ($G_1$)-($G_3$) and (\ref{e3.1}) hold. Then the following results are right.\\
(1) $m_{R_0}(\rho)$ is attained.\\
(2) If $\bar{u}\in \mathcal{U}_{R_0}(\rho)$ such that $E(\bar{u})=m_{R_0}(\rho)$, then $\bar{u}\in \mathcal{S}_\rho$ and there exists $\lambda_{\bar{u}}>0$, such that $(\bar{u},\lambda_{\bar{u}})$ is a solution to (\ref{e1.3}).
\end{theorem}

\begin{proof}
By Lemma \ref{lem3.3}, we can choose $\{u_n\}\subset \mathcal{U}_{R_0}(\rho)$ such that $E(u_n)\to m_{R_0}(\rho)$. By Lemma \ref{lem3.6},  there exists $\bar{u}\in \mathcal{D}_{\rho}\backslash\{0\}$  with $\|\nabla \bar{u}\|_2\leq R_0$ such that, replacing $u_n(\cdot+y_n)$ with $u_n$ and up to a subsequence, $u_n\rightharpoonup \bar{u}$  in $H^1(\mathbb{R}^N)$ and $u_n\to \bar{u}$ in $L^2(\mathbb{R}^N)$ and a.e. in $\mathbb{R}^N$, and, by Remark \ref{rmk3.1} and the definition of $m_{R_0}(\rho)$, $E(\bar{u})\geq m_{R_0}(\rho)$.

Denote $v_n:=u_n-\bar{u}$. From the weak convergence,
\begin{equation*}
\lim_{n\to+\infty}(\|\nabla u_n\|_2^2-\|\nabla v_n\|_2^2)=\|\nabla \bar{u}\|_2^2>0,
\end{equation*}
thus $\|\nabla v_n\|_2<R_0$ for $n\gg1$. Moreover, from the Bresiz-Lieb Lemma,
\begin{equation*}
m_{R_0}(\rho)=\lim_{n\to+\infty}E(u_n)=\lim_{n\to+\infty}E(v_n)+E(\bar{u})\geq \lim_{n\to+\infty}E(v_n)+m_{R_0}(\rho),
\end{equation*}
which implies that $\lim_{n\to+\infty}E(v_n)\leq 0$. This, (\ref{e3.11}), $h(\rho,R_0)=0$ and $\lim_{n\to+\infty}\|v_n\|_2=0$ give that
\begin{equation*}
\begin{split}
0\geq \lim_{n\to+\infty}E(v_n)&\geq \limsup_{n\to+\infty}\left(\frac{1}{2}\|\nabla v_n\|_2^2-\frac{1}{2}C_0^2S_1^{-\frac{N+\alpha}{N-2}}\|\nabla v_n\|_2^{2\frac{N+\alpha}{N-2}} \right)\\
&\geq C_1\rho^{2\frac{N+\alpha}{N}}R_0^{-2}\limsup_{n\to+\infty}\|\nabla v_n\|_2^2\geq 0,
\end{split}
\end{equation*}
which implies that $\limsup_{n\to+\infty}\|\nabla v_n\|_2=0$, that is, $u_n\to \bar{u}$ in $H^1(\mathbb{R}^N)$. Then $E(\bar{u})= \lim_{n\to+\infty} E(u_n)= m_{R_0}(\rho)$ and so, from Remark \ref{rmk3.1}, $\|\nabla \bar{u}\|_2 < R_0$. This proves (1).

Now we prove (2). Let $\bar{u}\in \mathcal{U}_{R_0}(\rho)$ such that $E(\bar{u})=m_{R_0}(\rho)$. Assume by contradiction that $a:=\|\bar{u}\|_2<\rho$. By Remark \ref{rmk3.2} (2), $m_{R_0}(a)>m_{R_0}(\rho)$.  Since $\bar{u}\in \mathcal{U}_{R_0}(a)$, we have $m_{R_0}(a)\leq E(\bar{u})=m_{R_0}(\rho)$. That is an contradiction. So we must have $\bar{u}\in \mathcal{S}_\rho$.

There exists $\lambda_{\bar{u}}\in \mathbb{R}$ such that
\begin{equation*}
-\Delta \bar{u}+\lambda_{\bar{u}} \bar{u}=
(I_{\alpha}\ast F(\bar{u}))f(\bar{u}),\quad \text{in}\  \mathbb{R}^N.
\end{equation*}
If $\lambda_{\bar{u}}\leq 0$, then by the Poho\v{z}aev identity (see \cite{Li-Ma 2020}), we obtain that
\begin{equation*}
\begin{split}
\frac{N-2}{2}\int_{\mathbb{R}^N}|\nabla \bar{u}|^2dx&\geq \frac{N-2}{2}\int_{\mathbb{R}^N}|\nabla \bar{u}|^2dx+\frac{N}{2}\lambda_{\bar{u}}\int_{\mathbb{R}^N}| \bar{u}|^2dx\\
&=\frac{N+\alpha}{2}\int_{\mathbb{R}^N}(I_\alpha\ast F(\bar{u}))F(\bar{u})dx.
\end{split}
\end{equation*}
Hence,
\begin{equation*}
\begin{split}
E(\bar{u})&=\frac{1}{2}\int_{\mathbb{R}^N}|\nabla \bar{u}|^2dx-\frac{1}{2}\int_{\mathbb{R}^N}(I_\alpha\ast F(\bar{u}))F(\bar{u})dx\\
&\geq \frac{1}{2}\int_{\mathbb{R}^N}|\nabla \bar{u}|^2dx-\frac{N-2}{N+\alpha}\frac{1}{2}\int_{\mathbb{R}^N}|\nabla \bar{u}|^2dx\\
&=\frac{1}{2}\frac{2+\alpha}{N+\alpha}\int_{\mathbb{R}^N}|\nabla \bar{u}|^2dx\geq 0,
\end{split}
\end{equation*}
which contradicts $E(\bar{u})=m_{R_0}(\rho)<-\frac{1}{2}b^2S_2^{-\frac{N+\alpha}{N}}\rho^{2\frac{N+\alpha}{N}}$. Thus $\lambda_{\bar{u}}>0$.
\end{proof}

\textbf{Proof of Theorem \ref{thm1.1}}. It is a direct result of  Theorem \ref{thm3.1}.

\textbf{Proof of Theorem \ref{thm1.2}}. We shall prove that
\begin{equation}\label{e3.12}
m_{R_0}(\rho)=\inf\{E(u):u\in \mathcal{S}_{\rho}\ \text{and}\ E|_{\mathcal{S}_{\rho}}'(u)=0\}.
\end{equation}
Then Theorem \ref{thm3.1} will yield that such an infimum is actually a minimum and  Theorem \ref{thm1.2} follows.

Now, we prove (\ref{e3.12}).  Since, from Theorem \ref{thm3.1}, there exists $\bar{u}\in \mathcal{S}_{\rho}$ with $\|\nabla  \bar{u}\|_2<R_0$ such that $E(\bar{u})=m_{R_0}(\rho)$ and $\bar{u}\in \{u:u\in \mathcal{S}_{\rho}\ \text{and}\ E|_{\mathcal{S}_{\rho}}'(u)=0\}$. So we have
\begin{equation*}
m_{R_0}(\rho)\geq \inf\{E(u):u\in \mathcal{S}_{\rho}\ \text{and}\ E|_{\mathcal{S}_{\rho}}'(u)=0\}.
\end{equation*}
Assume by contradiction that $m_{R_0}(\rho)>\inf\{E(u):u\in \mathcal{S}_{\rho}\ \text{and}\ E|_{\mathcal{S}_{\rho}}'(u)=0\}$, then there exists $u\in \mathcal{S}_{\rho}$ such that $E|_{\mathcal{S}_{\rho}}'(u)=0$ and $E(u)<m_{R_0}(\rho)$. From the definition of $m_{R_0}(\rho)$,  there holds $\|\nabla u\|_2\geq R_0$. In fact, we know from Lemma \ref{lem3.1}, Lemma \ref{lem3.2} (g2) and Lemma \ref{lem3.3} that $\|\nabla u\|_2> R_1$.

Consider the function $\varphi:(0,+\infty)\to \mathbb{R}$, $\varphi(\tau):=E(u_\tau)$. By the Nehari identity and the Poho\v{z}aev identity, every critical  point of $E|_{\mathcal{S}_{\rho}}$ belongs to
\begin{equation*}
\mathcal{M}:=\left\{u\in \mathcal{S}_{\rho}:\|\nabla u\|_2^2=\frac{N}{2}\int_{\mathbb{R}^N}(I_\alpha\ast F(u))(f(u)u-\frac{N+\alpha}{N}F(u))dx\right\}.
\end{equation*}
Thus, $\varphi'(1)=0$. By Lemma \ref{lem3.1},
\begin{equation}\label{e3.15}
\varphi(\tau)=E(u_\tau)\geq -\frac{1}{2}b^2S_2^{-\frac{N+\alpha}{N}}\rho^{2\frac{N+\alpha}{N}}+h(\rho,\|\nabla u_\tau\|_2)\|\nabla u_\tau\|_2^2.
\end{equation}
By  Lemma  \ref{lem3.2} (g2),  and noting that $\|\nabla u_\tau\|_2=\tau\|\nabla u\|_2$, we  have
\begin{equation}\label{e2.3}
h(\rho,\|\nabla u_\tau\|_2)>0\ \text{for}\ \tau\in({R_0}/{\|\nabla u\|_2},{R_1}/{\|\nabla u\|_2}).
\end{equation}
Combining (\ref{e3.15}) and (\ref{e2.3}), we obtain that
\begin{equation*}
\varphi(\tau)> -\frac{1}{2}b^2S_2^{-\frac{N+\alpha}{N}}\rho^{2\frac{N+\alpha}{N}},\ \text{for}\ \tau\in({R_0}/{\|\nabla u\|_2},{R_1}/{\|\nabla u\|_2}).
\end{equation*}
By the choice of $u$ and  Lemma \ref{lem3.3},
\begin{equation*}
\varphi(1)=E(u)<m_{R_0}(\rho)<-\frac{1}{2}b^2S_2^{-\frac{N+\alpha}{N}}\rho^{2\frac{N+\alpha}{N}}.
\end{equation*}
Thus,
\begin{equation*}
\varphi(\tau)>\varphi(1),\ \text{for}\ \tau\in({R_0}/{\|\nabla u\|_2},{R_1}/{\|\nabla u\|_2}).
\end{equation*}
By the expression of $E(u_\tau)$:
\begin{equation*}
\begin{split}
E(u_\tau)&=-\frac{1}{2}b^2\int_{\mathbb{R}^N}(I_\alpha\ast|u|^{\frac{N+\alpha}{N}})|u|^{\frac{N+\alpha}{N}}dx\\
&\qquad +\frac{1}{2}\tau^2\left\{\|\nabla u\|_2^2-\frac{1}{2}\frac{1}{\tau^{N+\alpha+2}}\int_{\mathbb{R}^N}(I_\alpha\ast G(\tau^{\frac{N}{2}}u))G(\tau^{\frac{N}{2}}u)dx\right.\\
&\qquad\qquad\qquad\qquad\left.-b\frac{1}{\tau^{(N+\alpha)/2+2}}\int_{\mathbb{R}^N}(I_\alpha\ast |u|^{\frac{N+\alpha}{N}})G(\tau^{\frac{N}{2}}u)dx
\right\},
\end{split}
\end{equation*}
and arguing as the proof of Lemma \ref{lem3.3}, we know
\begin{equation*}
\varphi(\tau)=E(u_\tau)<-\frac{1}{2}b^2\int_{\mathbb{R}^N}(I_\alpha\ast|u|^{\frac{N+\alpha}{N}})|u|^{\frac{N+\alpha}{N}}dx\quad \text{for}\ \tau\ll1.
\end{equation*}
By the following expression of $E(u_\tau)$ and (\ref{e2.3}),  similarly to (\ref{e3.11}) and (\ref{e3.2}),
\begin{equation*}
\begin{split}
\varphi(\tau)=E(u_\tau)&=\frac{1}{2}\int_{\mathbb{R}^N}|\nabla u_\tau|^2dx\\
&\qquad-\frac{1}{2}\int_{\mathbb{R}^N}[I_\alpha\ast(b|u_\tau|^{\frac{N+\alpha}{N}}+G(u_\tau))][b|u_\tau|^{\frac{N+\alpha}{N}}+G(u_\tau)]dx\\
&\geq -\frac{1}{2}b^2\int_{\mathbb{R}^N}(I_\alpha\ast|u_\tau|^{\frac{N+\alpha}{N}})|u_\tau|^{\frac{N+\alpha}{N}}dx+h(\rho,\|\nabla u_\tau\|_2)\|\nabla u_\tau\|_2^2\\
&> -\frac{1}{2}b^2\int_{\mathbb{R}^N}(I_\alpha\ast|u|^{\frac{N+\alpha}{N}})|u|^{\frac{N+\alpha}{N}}dx
\end{split}
\end{equation*}
for $\tau\in({R_0}/{\|\nabla u\|_2},{R_1}/{\|\nabla u\|_2})$, which is lager than $\varphi(\tau)$ with $\tau\ll 1$.
Noting that ${R_1}/{\|\nabla u\|_2}<1$, we know from above $\varphi$ has a local maximum point $\tau_u\in (0,1)$. From $(G_4)$, $\varphi'(\tau)<0$ for $\tau\in(\tau_u,+\infty)$, in contradiction with $\varphi'(1)=0$. Thus we must have (\ref{e3.12}) and  the proof is complete.

\bigskip

\textbf{Acknowledgements.}  This work is supported by the National
Natural Science Foundation of China (No. 12001403).



\end{document}